\documentclass[11pt]{amsart}

\usepackage{graphicx}

\usepackage{graphics,color}

\usepackage{psfrag}

\newcommand{\PSLC}{{\bf {PSL}}(2,\C)}

\newcommand{\D} {\mathbb {D}}
\newcommand{\Ha}{{\mathbb {H}}^{2}}
\newcommand{\Ho}{ {\mathbb {H}}^{3}}
\newcommand{\R} {\mathbb {R} }
\newcommand{\Z} {\mathbb {Z}}
\newcommand{\N}{\mathbb {N}}
\newcommand{\C} {\mathbb {C}}

\newcommand{\Q} {\mathbb {Q}}

\newcommand{\wt}{\widetilde}
\newcommand{\wh}{\widehat}

\newcommand{\M}{{\bf {M}}^{3} }

\newcommand{\hl}{{\bf {hl}}}

\newcommand{\len}{{\bf l}}

\newcommand{\Col}{{\mathcal {C}}}

\newcommand{\TB}{T^{1}}

\newcommand{\cl}{{\mathcal {K}}}

\newcommand{\Tr}{{\mathcal {T}}}

\newcommand{\RE}{\operatorname{Re}}

\newcommand{\foot}{\operatorname{foot}}

\newcommand{\Deck}{{\operatorname {Deck}}}

\newtheorem{corollary}{Corollary}[section]
\newtheorem{theorem}{Theorem}[section]
\newtheorem{definition}{Definition}[section]
\newtheorem{lemma}{Lemma}[section]
\newtheorem{conjecture}{Conjecture}[section]

\newtheorem{proposition}{Proposition}[section]

\theoremstyle{remark}
\newtheorem*{remark}{Remark}

\begin{document}

\title[Counting surfaces] {Counting essential surfaces in a closed hyperbolic three manifold}

\author[Kahn]{Jeremy Kahn}
\thanks{JK was supported by NSF grant DMS 0905812}
\address{Mathematics Department \\ Stony Brook University \\ Stony Brook, NY 11794}
\email{kahn@math.sunysb.edu}

\author[Markovic]{Vladimir Markovic}
\address{University of Warwick \\ Institute of Mathematics \\ Coventry,
CV4 7AL, UK}
\email{v.markovic@warwick.ac.uk}

\today


\begin{abstract} Let $\M$ be a closed hyperbolic three manifold.  We show that the number of genus $g$ surface subgroups of $\pi_1(\M)$ grows like  
$g^{2g}$. 

\end{abstract}

\maketitle

\section{Introduction} Let $\M$ be a closed hyperbolic 3-manifold and let $S_g$ denote a closed surface of genus $g$. Given a continuous mapping $f:S_g \to \M$ we let $f_{*}:\pi_1(S_g) \to \pi_1(\M)$ denote the induced homomorphism. 

\begin{definition} We say that $G<\pi_1(\M)$ is a surface subgroup of genus $g \ge 2$ is there exists a continuous map    $f:S_g \to \M$ such that the induced homomorphism $f_*$ is injective and $f_{*}(\pi_1(S_g))=G$. Moreover, the subsurface $f(S_g) \subset \M$ is said to be an essential subsurface. 
\end{definition}

Recently, we showed \cite{kahn-markovic-1} that every closed hyperbolic 3-manifold $\M$ contains an essential subsurface and consequently $\pi_1(\M)$ contains a surface subgroup.  
It is therefore natural to consider the question: How many conjugacy classes of surface subgroups of genus $g$ there are in $\pi_1(\M)$? 
This has already been considered by Masters  \cite{masters}, and our approach to this question builds on our previous work and improves on the work by Masters. 

Let $s_2(\M,g)$ denote the number of conjugacy classes of surface subgroups of genus at most $g$. We say that two surface subgroups $G_1$ and $G_2$ of $\pi_1(\M)$ are commensurable if $G_1 \cap G_2$ has a finite index in both $G_1$ and $G_2$.  Let $s_1(\M,g)$ denote the number  surface subgroups of genus at most $g$, modulo the equivalence relation of commensurability. Then clearly $s_1(\M,g) \le s_2(\M,g)$.
The main result of this paper is the following theorem.

\begin{theorem}\label{thm-main} Let $\M$ be  a closed hyperbolic 3-manifold. There exist two constants $c_1,c_2>0$   such that 
$$
(c_1 g)^{2g} \le s_1(\M,g) \le s_2(\M,g) \le (c_2g)^{2g},
$$
for  $g$ large enough. The constant $c_2$ depends only on the injectivity radius of $\M$.
\end{theorem}

In fact, Masters shows that 

$$
s_2(g,\M)<g^{c_{2}g}
$$
for some $c_2 \equiv c_2(\M)$, and likewise for some  $c_1 \equiv c_1(\M)$

$$
g^{c_{1}g}<s_1(g,\M)
$$
when $\M$ has a self-transverse totally geodesic subsurface. We follow Masters' approach to the upper bound, improving it
from $g^{c_{2}g}$ to $(c_2 g)^{2g}$ by more carefully counting the number of suitable triangulations of a genus $g$ surface.
Using our previous work \cite{kahn-markovic-1}  we replace Masters' conditional lower bound with an unconditional
one, and we improve it from $g^{cg}$ to $(c_1 g)^{2g}$ with the work of Muller and Puchta \cite{muller-puchta} counting number of
maximal surface subgroups of a given surface group. We then make new subgroup from old in the spirit of Masters' construction, but
taking the nearly geodesic subgroup from \cite{kahn-markovic-1}  as our starting point.

The above theorem enables us to determine the order of the number of surface subgroups up to genus $g$. We have the following corollary.

\begin{corollary}\label{cor-main} We have
$$
\lim_{g \to \infty} {{\log s_1(\M,g)}\over{2g \log g}}=\lim_{g \to \infty} {{\log s_2(\M,g)}\over{2g \log g}}=1.
$$
\end{corollary}

We make the following conjecture.

\begin{conjecture} For a given closed hyperbolic 3-manifold $\M$,  there exists a constant $c(M)>0$ such that 
$$
\lim_{g \to \infty} {{1}\over{g}} \sqrt[2g]{s_i(\M,g)}=c(M), \, i=1,2.
$$
\end{conjecture}





\section{The upper bound} Fix a closed hyperbolic 3-manifold $\M$. In this section we prove the upper bound in Theorem \ref{thm-main}, that is we show 

\begin{equation}\label{upper}
s_2(\M,g) \le (c_2g)^{2g}, 
\end{equation}
for some constant $c_2>0$.

\subsection{Genus $g$ triangulations} We have the following definition.

\begin{definition} Let   $S_g$ denote a closed surface of genus $g$. We say that a connected graph $\tau$ is a triangulation of genus g if it can be embedded into the surface $S_g$ such that every component of the set $S_g \setminus \tau$ is a triangle. The set of genus $g$ triangulations is denoted by $\Tr(g)$. We say that  $\tau \in \Tr(k,g) \subset \Tr(g)$ if: 
\begin{itemize}
\item each vertex of $\tau$ has the degree at most $k$, 
\item the graph $\tau$ has at most $kg$ vertices and edges.
\end{itemize}
\end{definition}

We observe that any given genus $g$ triangulation $\tau$, can be in a unique way (up to a homeomorphism of $S_g$) be embedded in $S_g$.

We say that Riemann surface is $s$-thick is its injectivity radius is bounded below by $s>0$. Every thick Riemann surface has a good triangulation.

\begin{lemma}\label{triangulation} Let $S$ be an $s$-thick Riemann surface of genus $g \ge 2$.  Then there exists $k=k(s)>0$ and a triangulation $\tau \in \Tr(k,g)$
that embeds in $S$,  such that
\begin{enumerate}
\item Every edge of $\tau$ is a geodesic arc of length at most $s$,
\item The triangulation $\tau$ has at most $kg$ vertices and edges,
\item The degree of each vertex is at most $k$.

\end{enumerate}

\end{lemma}

\begin{proof} Choose a a maximal collection of disjoint open balls in $S$ of radius ${{s}\over{4}}$. Let $V$ denote the set of centers of the balls from the collection. 
We may assume that no four points from $V$ lie on a round  circle (we always  reduce the radius of the balls by a small amount and move them into a general position). We construct the Delaunay triangulation associated to the set $V$ as follows. We connect two points from $V$ with the shortest geodesic arc between them, providing they belong to the boundary of a closed ball in $S$ that does not contain any other point from $S$. This gives an embedded graph $\tau$. Since no four points from $V$ lie on the same circle the graph $\tau$ is a triangulation.
It is elementary to check that $\tau$ has the stated properties, and we leave it to the reader.
\end{proof}

Given any injective immersion of $g:S_g \to \M$,  we can find a genus $g$ hyperbolic surface $S$, and a map $f:S \to \M$   homotopic to $g$, such that $f(S)$ is a pleated surface. Then $f$ does not increase the hyperbolic distance. Let $s$ denote the injectivity radius of $\M$. It follows that the injectivity radius of $S$ is bounded below by  $s$. We choose  a triangulation $\tau(S)$ of $S$ that satisfies the conditions in Lemma \ref{triangulation}.

Let $\Col=\{C_1,...,C_m \}$ be a finite collection of balls of radius ${{s}\over {4}}$ that covers $\M$. We may assume that $\Col$ is a minimal collection, that is, if we remove a ball from $\Col$, the new collection of balls does not cover $\M$. Let $f_i:S_i \to \M$, $i=1,2$, be two pleated maps,
and denote by $\tau(S_1)$ and $\tau(S_2)$ the corresponding triangulations of genus $g$ surfaces $S_1$ and $S_2$. If the genus $g$ triangulations $\tau(S_1)$ and $\tau(S_2)$ are identical, there
exists a homeomorphism $h:S_1 \to S_2$ such that $h(\tau(S_1))=\tau(S_2)$. Assume in addition that for every vertex $v$ of $\tau(S_1)$, the points $f_1(v)$ and $f_2(h(v))$ belong to the same ball $C_i \in \Col$. Then by  Lemma 2.4 in \cite{masters}, the maps $f_1$ and $f_2$ are homotopic.

Since  the set $\Col$ has $m$ elements, there are at most $m$ ways of mapping a given vertex of $\tau$ to the set $\Col$. Choose a vertex $v_1$ of $\tau$ and choose an image of $v_1$ in $\Col$, say $v_1$ is mapped to $C_1$. Let $v_1$ be a vertex of $\tau$, such that $v_0$ and $v_1$ are the endpoints of the same edge. Since each edge of $\tau$ has the length at most $s$, and the balls from $\Col$ have the radius ${{s}\over{4}}$. Since $f$ does not increase the distance, and $\Col$ is a minimal cover of $\M$, it follows that $v_1$ can be mapped to at most $K$ elements of $\Col$, where $K$ is a constant that depends only on $s$. Repeating this analysis yields the following estimate:

\begin{equation}\label{upper-1}
\wt{s}_2(\M,g) \le m K^{kg-1} |\Tr(k,g)|,
\end{equation}
where $\wt{s}_2(\M,g)$ denotes the number of conjugacy classes of surface subgroups of genus equal to $g$.

Let $\nu(k,n)$ denote the set of all graphs on $n$ vertices so that each vertex has the degree at most $k$.
Then  $|\Tr(k,g)| \le |\nu(k,kg)|$. 

\begin{remark} Observing the estimate 
$$
|\nu(k,n)|\le n^{kn},
$$
Masters showed
$$
\wt{s}_2(\M,g) \le g^{Dg},
$$
for some constant $D>0$. However, the set $\nu(k,kg)$ has many more elements than the set $\Tr(k,g)$. 
\end{remark}

The following  lemma will be proved in the next subsection.

\begin{lemma}\label{lemma-upper} There exists a constant $C>0$ that depends only on $k$,  such that for $g$ large we have 
$$
|\Tr(k,g)| \le (Cg)^{2g}.
$$
\end{lemma}

Given this lemma we now prove estimate (\ref{upper}).  It follows from the Lemma  \ref{lemma-upper} that for every $g$ large we have
$$
|\Tr(k,g)| \le (Cg)^{2g}.
$$
Combining this with (\ref{upper-1}) we get
$$
\wt{s}_2(\M,g) \le mK^{kg-1}(Cg)^{2g} \le (C_1 g)^{2g},
$$
holds for every $g \ge 2$, for some constant $C_1$. 
Then

\begin{align*}
s_2(\M,g)
&=\sum\limits_{r=2}^{g} \wt{s}_2(\M,r) \\
&=\sum\limits_{r=2}^{g}   (C_1 r)^{2r}  \\
& \le (c_2g)^{2g},
\end{align*}
for some constant $c_2$. This proves the estimate (\ref{upper}).

\subsection{The proof of Lemma \ref{lemma-upper} }

Fix a triangulation $\tau \in \Tr(k,g)$ and denote the set of oriented edges by $E(\tau)$.  
Let $\Q E(\tau)$ denote the vector space of all formal sums (with rational coefficients) of edges from $E(\tau)$. 

Choose a spanning tree $T$ (a spanning tree of a connected graph is a connected tree that contains all of its vertices) for $\tau$.
Let $H_1(S_g)$ denote the first homology with rational coefficients of the surface $S_g$. We define the linear map $\phi:\Q E(\tau) \to H_1(S_g)$ as follows.
Let $e \in (E(\tau) \setminus T)$. Then the union $e \cup T$ is homotopic (on $S_g$) to a unique (up to homotopy) simple closed curve $\gamma_e \subset S_g$. 
We let $\phi(e)$ denote the homology class of the curve $\gamma_e$ in  $H_1(S_g)$. We extend the map $\phi$ to $\Q E(\tau)$ by linearity.

Denote the kernel of $\phi$ by $K(\phi)$ and set
$$
H_1(\tau,T)={{ \Q E(\tau)}\over{K(\phi)}}.
$$
Then the quotient map (also denoted by) $\phi:H_1(\tau,T) \to H_1(S_g)$ is injective, and in fact it an isomorphism.  Since $\tau$ is a genus $g$ triangulation,  
the embedding of the triangulation $\tau$ to $S_g$ induces the surjective map of the fundamental group of $\tau$ to the fundamental group of $S_g$. Then the induced map $\phi$ between the corresponding homology groups is injective.
 
Let $e_1,...,e_{2g} \in E(\tau)$ denote a set of $2g$ edges whose equivalence classes generate $H_1(\tau,T)$. 

\begin{lemma}\label{lemma-h} Let $X=T \cup \{e_1,...,e_{2g} \}$. Then every component of the set $S_g \setminus X$ is simply connected.
\end{lemma}

\begin{proof} The set $X$ is connected (since it contains the spanning tree $T$, and the tree $T$ contains all the vertices). 
Suppose that there exists a component of the set $S_g \setminus X$ that is not simply connected. Then there exists a simple closed curve $\gamma \subset S_g$ that is not 
homotopic to a point, and such that 
$$
\gamma \cap  X  =\emptyset.
$$

If $\gamma$ is a non-separating curve then the homology class of  $\gamma$ is non-trivial in $H_1(S_g)$. Therefore, there exists  a non-separating simple closed $\alpha \subset S_g$ that  intersects the curve $\gamma$ exactly once. Let $q_1,...,q_{2g} \in \Q$ be such that 
$$
\phi(q_1e_1+...+q_{2g}e_{2g})=[\alpha],
$$
where $[\alpha] \in H_1(S_g)$ denotes the homology class of $\alpha$.
Since the intersection pairing between $[\alpha]$ and $[\gamma]$ is non-zero, and $\phi(e_1),...,\phi(e_{2g})$ is a basis for $H_1(S_g)$, we conclude that for some $i \in \{1,...,2g\}$, 
the curve $\gamma$ intersects $e_i \cup T$, which is a contradiction.

Suppose that  $\gamma$ is a separating curve and denote by $A_1$ and $A_2$ the two components of the set  $S_g \setminus \gamma$. The set $X$ is connected, and by the assumption it does not intersect
$\gamma$. This implies that $X$ is contained in one of the two sub-surfaces $A_i$, say $X \subset A_1$. Then $X \cap A_2=\emptyset$.

Since $\gamma$ is not homotopic to a point, each $A_i$ is a non-planar surface with one boundary component. Therefore, the subsurface $A_2$ contains  a non-separating simple closed curve $\gamma_2$. Then $\gamma_2$ is a non-separating simple closed curve in $S_g$ by the above argument we have that $\gamma_2$ intersects  the set $X$. This is a contradiction since $X \cap A_2=\emptyset$.

\end{proof}

Let $P_1,...,P_l$ denote the components of the set $S_g \setminus X$. Each $P_i$ is a polygon and we let $m_i$ denote the number of sides of the polygon $P_i$. 
Since each edge in $X$ can appear as a side in at most two such polygons, we have the inequality

\begin{equation}\label{sides}
\sum\limits_{i=1}^{l}m_i \le 2kg,
\end{equation}
since by definition the triangulation $\tau$ has at most $kg$ edges.

We  proceed to prove Lemma \ref{lemma-upper}.  We can obtain every triangulation $\tau \in \Tr(k,g)$ as follows. We first choose a spanning tree $T$, which is a tree that has at most $kg$ vertices. 
Then to the tree $T$ we add $2g$ edges
$e_1,...,e_{2g}$ in an arbitrary way. After adding the edges, at each vertex of the graph $T \cup \{e_1,...,e_{2g} \}$ we choose a cyclic ordering. We thicken the edges of the graph $T \cup \{e_1,...,e_{2g} \}$ to obtain the ribbon graph and the corresponding  surface $R$  with boundary (if this surface does not have genus $g$ we discard this graph). The boundary components of the surface 
$R$ are  polygonal curves  $P_i$, $i=1,..,l$,  made out of the edges from  $T \cup \{e_1,...,e_{2g} \}$. We then choose a triangulation of each  polygon $P_i$.

It follows from this description that we can bound the number of triangulations from $\Tr(k,g)$ by
$|\Tr(k,g)|\le abcd$, where

$$
a=\{\text{number of unlabelled trees}\, T \, \text{with}\,  n \le kg \,\, \text{vertices} \} ,
$$

$$ 
b=\{ \text{number of ways of adding}\,  2g \, \text{unlabelled edges}\,\, e_1,...e_{2g} \, \text{to}\, T \} ,
$$

$$
c= \{ \text{number of  cyclic orderings of edges of}\, T \cup \{e_1,...,e_{2g} \} \},
$$

$$
d= \{ \text{number of triangulations  of the polygons}\,\,  P_i \}.
$$

Let $t(n)$ denote the number of different unlabelled trees on $n$ vertices. By \cite{bender-canfield} we have $t(n) \le C 12^{n}$, for some universal constant $C>0$. It follows that
$a \le 2C12^{kg}$. The tree $T$ has at most $kg$ edges, so there are at most $(kg)^{2}$ ways of adding a labelled edge to $T$. All together there are at most $(kg)^{4g}$ ways of adding a labelled collection of $2g$ edges to $T$. To obtain the number of ways of adding unlabelled collection of $2g$ edges we need to divide this number by $(2g)!$. This yields the estimate 
$$
b \le {{(kg)^{4g} }\over {(2g)!}} <( k^{2} g)^{2g},
$$
for $g$ large.  

Since each vertex of $\tau$ has the degree at most $k$, and $\tau$ has at most $kg$ edges, we obtain the estimate 
$$
c \le (k!)^{kg}.
$$

Let $p(m)$ denote the number of triangulations of a polygon with $m$ sides. Then $p(m)$ is the $(m-2)$-th Catalan number and we have $p(m) < 2^{2m}$.  
As above,  let $P_1,...,P_l$ denote the polygons that we need to triangulate and let $m_i$ denote the number of sides of the polygon $P_i$. 
Then 
$$
d \le \max \Pi_{i=1}^{l} p(m_i) \le  \max \le 4^{m_{1}+...+m_{l} } ,
$$
where the maximum is taken over all possible vectors $(m_1,...,m_l)$, $1 \le l \le 2kg$, such that  $m_1+...+m_l \le 2kg$ (see estimate (\ref{sides}) above).
But since  $m_1+...+m_l \le 2kg$ we have $d \le 4^{2kg}$.

Putting the estimates for $a,b,c,d$ together we prove the lemma.

\begin{remark} If we are given a tree on a surface $S$, along with $2g$ edges connecting the vertices of
the tree (and satisfying the hypothesis of Lemma \ref{lemma-h})  and a map of the resulting graph into $\M$,
the we can determine the map of $S$ into $\M$, up to homotopy. Thus we need only bound $|\Tr'(k,g)|$, where 
$\Tr'(k,g)$ is the set of trees of size at most $kg$, with $2g$ more edges added; we observe that $|\Tr'(k,g)|<ab$.
\end{remark}

\section{Quasifuchsian representations of surface groups}

\subsection{Generalized pants decomposition and the Complex Fenchel-Nielsen coordinates}

For background on complex Fenchel-Nielsen coordinates see \cite{series}, \cite{kour}, \cite{ser-tan}, \cite{kahn-markovic-1}. The exposition and notation we use here is in line with Section 2 in 
\cite{kahn-markovic-1}.

Let $X$  a compact topological surface (possibly with boundary)  and let $\rho:\pi_1(X) \to \PSLC$ be a representation (a homomorphism). 
We say that $\rho$ is a $K$-quasifuchsian representation if the group $\rho(\pi_1(X))$ is $K$-quasifuchsian, in which case we can 
equip $X$ with a complex structure  $X=\Ha /F$, for some Fuchsian group $F$,  such that   $f_*=\rho \circ \iota$. Here $\iota: F \to \pi_1(X)$ is an isomorphism, and  
$f_*:F \to fFf^{-1}$ is the conjugation homomorphism, induced by an equivariant $K$-quasiconformal map  $f: \partial{\Ho} \to \partial{\Ho}$.

We will also say that a quasisymmetric map $f: \partial{\Ha} \to \partial{\Ho}$ is $K$-quasiconformal if it has a $K$-quasiconformal extension to $\partial{\Ho}$.

By $\Pi$ we denote a topological pair of pants with cuffs $C_i$, $i=1,2,3$. Recall that that to every representation  $\rho:\pi_1(\Pi) \to \PSLC$, we associate the three half lengths 
$\hl(C_i) \in \C_{+} / 2i\pi \Z$, where $\C_{+}= \{ z \in \C: \, \RE(z)>0\}$. If $\rho$ is quasifuchsian then it is uniquely determined by the half lengths. 
The conjugacy class $[\rho]$  of a quasifuchsian representation $\rho$ is called a skew pair of pants.

We let $\Pi$ and $\Pi'$ denote two pairs of pants and  let $\rho:\pi_1(\Pi) \to \PSLC$ and   $\rho':\pi_1(\Pi') \to \PSLC$ denote two representations. 
Suppose that for some $c_1 \in \pi_1(\Pi)$ and $c'_1 \in \pi_1(\Pi')$, that belong to the  conjugacy classes of $C_1$ and $C'_1$ respectively, we have  $\rho(c_1)=\rho'(c'_1)$, and  $\hl(C_1)=\hl(C'_1)$. 
By  $s(C) \in \C/ (\hl(C) \Z+ 2\pi i \Z)$ we denote the  reduced twist-bend parameter, which measures how the two skew pairs of pants $[\rho]$ and $[\rho']$ align together along the axis of the loxodromic transformation $\rho(c_1)=\rho'(c'_1)$.

A pair $(\wt{\Pi}, \chi)$ is a generalized pair of pants if $\wt{\Pi}$ is a compact surface with boundary and $\chi$ is a 
finite degree covering map $\chi:\wt{\Pi} \to \Pi$, where $\Pi$ is a pair of pants. (We will also call $\wt{\Pi}$ a generalized pair of pants if $\chi$ is understood.)  
By $\chi_{*}:\pi_1(\wt{\Pi}) \to \pi_1(\Pi)$ we denote an induced homomorphism.

\begin{definition}\label{def-1} Let  $(\wt{\Pi}, \chi)$ be a generalized pair of pants and 
$$
\wt{\rho}:\pi_1(\wt{\Pi}) \to \PSLC,
$$
be a representation. We say that $\wt{\rho}$ is admissible with respect to $\chi$ if it factors through $\chi_*$, that is there exists $\rho:\pi_1(\Pi) \to \PSLC$ such that $\wt{\rho}=\rho \circ \chi_{*}$.  
\end{definition}

Let $\wt{C}_j$, $j=1,...,k$, denote the cuffs (the boundary curves) of the surface $\wt{\Pi}$, and let $C_1,C_2,C_3$ continue to denote the cuffs of $\Pi$. 
Then  $\chi$ maps each $\wt{C}_j$ onto some  $C_i$ with some degree $m_j \in \N$. We say that such a curve $\wt{C}_j$ is a $degree$ $m_j$ curve.
For every admissible $\wt{\rho}$ we define the half length $\hl(\wt{C}_j)$ as
$\hl(\wt{C}_j)=\hl(C_i)$. Let $\wt{c_j} \in \pi_1(\wt{\Pi}^{0})$ be in the conjugacy class that corresponds to the cuff $\wt{C}_j$. Then 
$$
\len(\wt{\rho}(c_i))=2m_j\hl(C_i) \,\,  (\text{mod}(2\pi i \Z)).
$$

Let $S$ be an oriented closed topological surface with a generalized pants decomposition. By this we mean  that we are 
given   a collection $\Col$  of disjoint  simple closed curves on $S$, such that  
for every component $\wt{\Pi}$ of  $S \setminus \Col$ there  is an associated  finite cover $\chi:\wt{\Pi} \to \Pi$.  Let 
$$
\wt{\rho}:\pi_1(S) \to \PSLC
$$
be a representation. We make the following assumptions on $\rho$:

\begin{itemize}

\item Given a curve $C \in \Col$ there exists two (not necessarily different) generalized pairs of pants $\wt{\Pi}_1$ and $\wt{\Pi}_2$ that both contain 
$C$ as a cuff, and that lie on different sides of $C$. Let $\chi_1:\wt{\Pi}_1 \to \Pi_1$ and $\chi_2:\wt{\Pi}_2 \to \Pi_2$ be the corresponding finite covers, where $\Pi_1$ and $\Pi_2$ are two pairs of pants. We  assume that the restrictions of $\chi_1$ and $\chi_2$ on the curve $C$ are of the same degree. 

\item  For every generalized pair of pants $\wt{\Pi}$ from the above decomposition of $S$, the restriction  $\rho:\pi_1(\wt{\Pi}) \to \PSLC$ is admissible with respect to the covering map 
$\chi:\wt{\Pi} \to \Pi$  (in the sense of Definition \ref{def-1}). 

\item For  every  $C \in \Col$, the half lengths of $C$ coming from the representations  $\rho:\pi_1(\wt{\Pi}_1) \to \PSLC$ and $\rho:\pi_1(\wt{\Pi}_2) \to \PSLC$ are one and the same.
 
\end{itemize}

Continuing with the above notation, let $C_i \subset \Pi_i$ denote the cuff such that $\chi_i(C)=C_i$.  
Let $\rho_i:\pi_1(\Pi_i) \to \PSLC$, $i=1,2$, be the representations such that the restriction of $\rho$ to $\pi_1(\wt{\Pi}_i)$ is equal to $\rho_i \circ (\chi_i)_{*}$. 
We define the reduced twist bend parameter $s(C)$ associated to $\rho$  to be equal to the  reduced twist-bend parameter for the representations $\rho_1$ and $\rho_2$.

So given a closed surface $S$ with a generalized pants decomposition $\Col$, and a representation $\rho:\pi_1(S) \to \PSLC$, we have defined the 
parameters  $\hl(C) \in \C_{+} / 2k\pi \Z$ and $s(C) \in \C/ (\hl(C) \Z+ 2\pi i \Z)$. 
The collection of pairs $(\hl(C),s(C))$, $C \in \Col$, is called the reduced Fenchel-Nielsen coordinates.
We observe that a representation $\rho:\pi_1(S) \to \PSLC$ is Fuchsian if and only if all the coordinates  $(\hl(C),s(C))$ are real.

The following elementary proposition (see \cite{kahn-markovic-1}) states that although a representation $\rho:\pi_1(S) \to \PSLC$ 
is not uniquely determined by its reduced Fenchel-Nielsen coordinates, it can be in a unique way  embedded in a
holomorphic family of representations.

\begin{proposition}\label{prop-rep} 
Fix a closed topological surface $S$ with  a generalized pants decomposition $\Col$. 
Let $z \in \C^{\Col}_{+}$ and $w \in \C^{\Col}$ denote complex parameters.
Then there exists a holomorphic (in $(z,w)$)  family of representations 
$$
\rho_{z,w}:\pi_1(S) \to \PSLC,
$$
such that  $\hl(C)=z(C)$, (mod$(2\pi i \Z)$) and  $s(C)=w(C)$, (mod$(\hl(C) \Z+ 2\pi i \Z)$). Moreover, for any 
$(z_0,w_0) \in \C^{\Col}_{+} \times \C^{\Col}$, the family of  representations $\rho_{z,w}$ is uniquely determined by  
the representation $\rho_{z_{0},w_{0}}$.  
\end{proposition}

The  representation $\rho_{z,w}$ is Fuchsian if and only if both $z$ and $w$ are real, that is  $z \in \R^{\Col}_{+}$ and $w \in \R^{\Col}$. In this case the group $\rho_{z,w}(\pi_1(S))$ is of course discrete. Moreover, in \cite{kour} it has been proved that all quasifuchsian representations (up to conjugation in $\PSLC$) of $\pi_1(S)$ correspond to some neighborhood of the set $\R^{\Col}_{+}$ and $\R^{\Col}$ But in general, little is known for which choice of parameters $z,w$ the group $\rho_{z,w}(\pi_1(S))$ will be discrete. In the next subsection we prove the following  result in this direction.
Start with a nearly Fuchsian group $G<\PSLC$. We obtain a new group $G_1<\PSLC$ from $G$ by bending (by some definite angles) along some sparse equivariant collection of geodesics whose endpoints are in the limit set of $G$. Then  the new group $G_1$ is also quasifuchsian (although it is not nearly Fuchsian anymore).

\subsection{Small deformations of a sparsely bent pleated surface}

We let $S$ continue to denote  a closed surface with a generalized pants decomposition $\Col$, and we fix a holomorphic  family of representations 
$\rho_{z,w}$ as in Proposition \ref{prop-rep}. We set $G(z,w)=\rho_{z,w}(\pi_1(S))$.

Let $\Col_0 \subset \Col$ denote a sub-collection of curves. For  $z \in \R^{\Col}_{+}$ and $w \in \R^{\Col}$,
we let $S_{z,w}$ denote the Riemann surface isomorphic to  $\Ha / G(z,w)$, and on $S_{z,w}$ we identify the curves from $\Col$ with the corresponding  geodesics representatives. 
By $\cl(S_{z,w})$ we denote the largest number so that  the collection of collars (of width $\cl(S_{z,w})$) around the curves from $\Col_0$ is disjoint on  $S_{z,w}$.
For each $C \in \Col_0$, we choose   a number $-{{3}\over{4}}\pi < \theta_C<{{3}\over{4}} \pi$ (for each curve $C \in (\Col \setminus \Col_0)$ we set $\theta_C=0$).

The purpose of this subsection is to prove the following theorem. 

\begin{theorem}\label{thm-sparsely} There exist constants $K>1$ and $C>0$ such that the following holds. Let  $z_0 \in \R^{\Col}_{+}$ and $w_0 \in \R^{\Col}$, and $z_1 \in \C^{\Col}_{+}$ and $w_1 \in \C^{\Col}$ be such that the representation  $\rho=\rho_{z_{1},w_{1}} \circ \rho^{-1}_{z_{0},w_{0}}:G(z_0,w_0) \to G(z_1,w_1)$, is $K$-quasifuchsian. Set $z_2=z_1$ and $w_2=w_1+i\theta_C$. If $\cl(S_{z_{0},w_{0}}) \ge C$, then  the representation  $\rho_{z_{2},w_{2}}:\pi_1(S) \to \PSLC$ is $K_1$-quasifuchsian, where $K_1$ depends only on $K$ and $C$.

\end{theorem}

The following lemma is elementary. 

\begin{lemma}\label{lemma-elem} Let $0 \le \theta_0< \pi$ and $B_0 \ge 1$. There exist constants $L(\theta_0,B_0)>0$ and $C(\theta_0,B_0)>0$  such that the following holds. Let $I \subset \R$ be an interval that is partitioned into intervals $I_j$, $j=1,...,k$. Let $\psi:I \to \Ho$  be  a continuous map, such that $\psi$ maps each $I_j$ onto a geodesic segment and the  restriction of $\psi$ on $I_j$ is $B_0$-bilipschitz. Assume in addition that the bending angle between two consecutive geodesic intervals $\psi(I_j)$ and $\psi(I_{j+1})$  is at most $\theta_0$. If the length of every $I_j$ is at least $C(\theta_0,B_0)$ then $\psi$ is $L(\theta_0,B_0)$-bilipschitz.
\end{lemma}

Let $\psi:I \to \Ho$ be a $C^{1}$ map, where $I \subset \R$ is a closed interval. For $x \in I$ let $v(x) \in \TB I$ denote the unit vector that points toward $+\infty$. Let $\delta>0$. 
We say that the map $\psi$ is $\delta$-nearly geodesic if for every $x,y \in I$ such that $x <y \le x+ 1$, we have that the angle between the vector $\psi_{*}(v(x))$ and the oriented geodesic segment from $\psi(x)$ to $\psi(y)$ is at most $\delta$. 

Clearly, every $0$-nearly geodesic map is an isometry, and a sequence of $\delta_n$-nearly geodesic maps converges (uniformly on compact sets) in the $C^1$ sense  to an isometry, when $\delta_n \to 0$.  
The following lemma is a generalization of the previous one.

\begin{lemma}\label{lemma-elem-1}  There exist universal  constants $L,C, \delta>0$,  such that the following holds. Suppose that $I$  is partitioned into intervals $I_j$, $j=1,...,k$, and let $\psi:I \to \Ho$ be a continuous  map, whose restriction on every closed sub-interval $I_j$ is $C^{1}$ and  $\delta$-nearly geodesic. Assume that the bending angle between two consecutive curves $\psi(I_{j})$ and $\psi(I_{j+1})$ is at most ${{3}\over{4}}$ (by the bending angle between two $C^{1}$ curves we mean the appropriate angle determined by the two tangent vectors at the point where the two curves meet).  
If the length of every $I_j$ is at least $C$ then $\psi$ is $L$-bilipschitz.
\end{lemma}

\begin{proof} Choose any two numbers ${{3}\over{4}}< \theta_0 <\pi$ and $B_0>1$. Assuming that $C > C(\theta_0,B_0)$  we can partition each $I_j$ into sub-intervals of length between $C(\theta_0,B_0)$ and $2C(\theta_0,B_0)$. Replacing each $I_j$ with these new intervals we obtain the new partition of $I$ into intervals $J_i$, where each $J_i$ has the length between $C(\theta_0,B_0)$ and  $2C(\theta_0,B_0)$. Let $\psi:I \to \Ho$ be the continuous map that agrees with $\psi$ at the endpoints of all intervals $J_i$, and such that the restriction of $\psi$ to each $J_i$ maps $J_i$ onto a geodesic segment in $\Ho$, and  is affine (the map $\psi$ either stretches or contracts distances by a constant factor on a given $J_i$).

Next, since we have the upper bound $2C(\theta_0,B_0)$ on the length of each interval $J_i$, we can choose $\delta>0$ small enough such that the bending angle between two consecutive geodesic segments $\phi(J_i)$ and $\phi(J_{i+1})$ is at most $\theta_0$. Also, by choosing $\delta$ small we can arrange that the map $\phi \circ \psi^{-1}$ is $2$-bilipschitz (the same statement holds if we  replace $2$ by any other number greater than $1$). By the previous lemma the map $\phi$ is 
$L(\theta_0,B_0)$-bilipschitz. Then the map $\psi$ is  $2L(\theta_0,B_0)$-bilipschitz. We take $L=2L(\theta_0,B_0)$, and $C=C(\theta_0,B_0)$, and the lemma is proved.
\end{proof}

We are now ready to prove Theorem \ref{thm-sparsely}.

\begin{proof}
Recall that $f:\partial{\Ha} \to \partial{\Ho}$ is a $K$-quasiconformal map that conjugates $G(z_0,w_0)$ to $G(z_1,w_1)$. Let  $\wt{f}:\Ha \to \Ho$ denote the  Douady-Earle extension  of  $f$. Then  $\wt{f}$ is $\delta$-nearly geodesic (this means that the restriction of $\wt{f}$ to every geodesic segment is  $\delta$-nearly geodesic in the sense of the above definition) for some $\delta=\delta(K)$, and $\delta(K) \to 0$, when $K \to 1$. 

If we assume that $\cl(S_{z_{0},w_{0}})$ is large enough, by adjusting $\wt{f}$, we can arrange that $\wt{f}$ is then $C^{\infty}$ mapping that  maps the geodesics in $\Ha$ that are lifts of the geodesics from $\Col_0$ onto the corresponding geodesics in $\Ho$,  and ensure that $\wt{f}$ is $2\delta$-nearly geodesic. Moreover, we can arrange  that $\wt{f}$ is conformal at every point of every  geodesic $\gamma$ that is a lift of a curve from $\Col_0$.

We construct the map $\wt{g}:\Ha \to \Ho$ that conjugates $G(z_0,w_0)$ to $G(z_{1},w_{1})$ as follows. Let $M$ be a component of the set $S_{z_{0},w_{0}} \setminus \Col_0$, and let $\wt{M} \subset \Ha$ denote its universal cover, that is $\wt{M}$ is an ideal polygon with infinitely many sides in $\Ha$, whose sides are lifts of the geodesics from $\Col_0$ that bound $M$. We set $\wt{g}=\wt{f}$ on $\wt{M}$.  

Let $\wt{M}_1 \subset \Ha$ be the universal cover of some other component $M_1$ of the set $S_{z_{0},w_{0}} \setminus \Col_0$. Let $\gamma$ denote a lift of a geodesic  $C \in \Col_0$, 
and assume that the polygons $\wt{M}$ and $\wt{M}_1$ are glued to each other along $\gamma$ (that is, $C$ is in the boundary of both $M$ and $M_1$). 
Let $R(\theta_C) \in \PSLC$, denote the rotation about $\wt{g}(\gamma)$ for the angle $\theta_C$.  We define  $\wt{g}$ on $\wt{M}_1$ by letting
$\wt{g}=R(\theta_C) \circ \wt{f}$.  We then define $\wt{g}$ inductively on the rest of $\Ha$.

Clearly $\wt{g}$ conjugates  $G(z_{0},w_{0})$ to $G(z,w)$. Let $x \in \gamma$, and  $v(x)$ a non-zero  vector that is orthogonal to $\gamma$. Since $|\theta_C| \le {{3}\over {4}} \pi$, and since $\wt{f}$ is differentiable at $x$, it follows that  the bending angle between the vectors $\wt{g}_{*}(v(x))$ and $\wt{g}_{*}(-v(x))$ is at most ${{3}\over{4}}\pi$. If $u(x)$ is any other vector at $x$, 
since $\wt{f}$ is conformal at $x$, it follows that the bending angle between the vectors $\wt{g}_{*}(u(x))$ and $\wt{g}_{*}(-u(x))$ is at most as big as the bending angle between the vectors $\wt{g}_{*}(v(x))$ and $\wt{g}_{*}(-v(x))$. Therefore, the restriction of the  map $\wt{g}$ on every geodesic segment satisfies the assumptions of  Lemma \ref{lemma-elem-1}. It follows that  $\wh{g}$ is 
$L$-bilipschitz, where $L$ depends only on $K$ and $C$. Therefore the representation $\rho_{z_{2},w_{2}}:\pi_1(S) \to \PSLC$ is $K_1$-quasifuchsian, where $K_1$ depends only on $K$ and $C$.

\end{proof}

\subsection{Convex hulls and pleated surfaces} In this subsection we digress from the notions of generalized pants decompositions and Fenchel-Nielsen coordinates, to prove a preliminary lemma about hyperbolic convex hulls of quasicircles.

Let $\lambda$ be a discrete geodesic lamination in $\Ha$, and let $\cl(\lambda)$ denote the largest number such that for every small $\epsilon>0$, the collection of collars (crescent in $\Ha$) of width $\cl(\lambda)-\epsilon$ around the leafs of $\lambda$ is disjoint in $\Ha$. Let $\mu$ denote a  real valued measure on $\lambda$. By $\iota_{\lambda,\mu}=\iota:\Ha \to \Ho$, we denote the corresponding pleating map.  As usual, by $\iota(\lambda)$ we denote the collection of geodesics in $\Ho$ that are images of geodesics from $\lambda$ under $\iota$. If the map $\iota$ is $L$-bilipschitz then $\iota$ extends continuously to a $K$-quasiconformal 
map $f:\partial{\Ha} \to \partial{\Ho}$, for some $K=K(L)$. In this case, let $W \subset \Ho$ denote the convex hull of the quasicircle $\iota(\partial{\Ha})$.
The convex hull $W$ has two boundary components which we denote by $\partial_1{W}$ and  $\partial_2{W}$.  We prove the following lemma.

\begin{lemma}\label{lemma-rigid} There  exist universal constants  $C_1, \delta_1>0$, with the following properties. Assume that $\cl(\lambda)>C_1$,  and that 
${{\pi}\over{4}} \le |\mu(l)| \le {{3\pi}\over{4}}$, for every $l \in \lambda$. Then for every geodesic  $\gamma \subset W$ the following holds:
\begin{enumerate}
\item If  $\gamma \in \iota(\lambda)$, then  for every point $p \in \gamma$, the inequality 
$$\max_{i=1,2} d(p,\partial_i{W})> \delta_1
$$ 
holds,  
\item If $\gamma$ does not belong to $\iota(\lambda)$, then for some point $p \in \gamma$, the inequality  $\max_{i=1,2} d(p,\partial_i{W})< {{\delta_1}\over{3}}$ holds.
\end{enumerate}

\end{lemma}

Compare this lemma with Lemma 4.2 in \cite{masters}.

\begin{proof} It follows from Lemma \ref{lemma-elem} that for $C_1$ large enough, the pleating map $\iota$ is $L$-bilipschitz for some universal constant $L>1$. Observe that $\iota(\Ha) \subset W$.  Moreover, there is a constant $M_0>0$, that depends only on $L$, such that for every $p \in W$ we have  $d(p,\iota(\Ha))< M_0$

We choose $\delta_1>0$ as follows. Let $P_0$ be the pleated surface in $\Ho$ that has a single bending line $\gamma_0$, and 
with the bending angle equal to ${{\pi}\over{4}}$. Then $P_0$ is bounded by a quasicircle at $\partial{\Ho}$. 
Denote by $W_0$ the convex hull of this quasicircle and let $\partial_i(W_0)$, $i=1,2$, denote the two boundary components of $W_0$. 
Then there exists $\delta_1>0$  such that for every point $p \in \gamma_0$, we have $\max_{i=1,2} d(p,\partial_i{W_0})> 2\delta_1$. Observe
that $\gamma_0$ belongs to exactly one of the convex hull boundaries $\partial_1{W_0}$ and  $\partial_2{W_0}$, so one of the numbers  
$d(p,\partial_1{W_0})$ and  $d(p,\partial_2{W_0})$ is zero and the other one is larger than $2\delta_1$.

Assume that the first statement of the lemma is false. Then there exists a sequence of measured laminations $(\lambda_n,\mu_n)$ with the property $\cl(\lambda_n) \to \infty$, and there are  geodesics $l_n \in \lambda_n$, and points $p_n \in \gamma_n=\iota_n(l_n)$, such that the inequality 
\begin{equation}\label{ass-1}
\max_{i=1,2} d(p_n,\partial_i{W_n}) \le \delta_1, 
\end{equation}
holds. We may assume that  $p_n=p$, and $\gamma_n=\gamma$, 
for every $n$, where $p$ and $\gamma$ are fixed. Since $\iota_n$ is $L$-bilipschitz, after passing to a subsequence if necessary, the sequence $\iota_n$ converges (uniformly on compact sets) to a pleating map $\iota_{\infty}$. The pleating map $\iota_{\infty}$  corresponds to the pleating surface $P_{\infty}$, that has a single bending line $\gamma_{\infty}$, with the bending angle at least ${{\pi}\over{4}}$. Then $W_n$ converges to $W_{\infty}$ uniformly on compact sets in $\Ho$, where $W_{\infty}$ is the convex hull of the quasicircle that bounds $P_{\infty}$. It follows that  $d(p_n,\partial_i{W_n}) \to d(p,\partial_i{W_{\infty}})$. We may assume that $\gamma_{\infty}=\gamma_0$, where $\gamma_0$ is the bending line of the pleated surface $P_0$ defined above. Then we have $\max_{i=1,2} d(p,\partial_i{W_{\infty}}) \ge \max_{i=1,2} d(p,\partial_i{W_0})> 2\delta_1$.  But this  contradicts (\ref{ass-1}).

We now prove the second statement of the lemma. Let $\gamma$ be a geodesic in $W$ that is not in $\iota(\lambda)$. 
Then  we can find a point  $p \in \gamma$, such that $d(p,\iota(\lambda))>\cl(\lambda)$.  
Assuming that the second statement is false, we again produce a sequence $\lambda_n$ with $\cl(\lambda_n) \to \infty$, and such that for 
some sequence of geodesics  $\gamma_n \subset W_n$, that do not belong to $\iota(\lambda_n)$, and all the points $p \in \gamma_n$, the inequality  
\begin{equation}\label{ass-2}
\max_{i=1,2} d(p,\partial_i{W_n}) \ge {{\delta_1}\over{3}}, 
\end{equation}
holds for $n$ large enough. By the previous discussion, there exists a sequence
of points $p_n \in \gamma_n$, such that $d(p_n,\iota_n(\lambda_n))>\cl(\lambda_n)$.

Let $q_n \in \iota_n(\Ha)$ be points such that $d(p_n,q_n)<M_0$, where $M_0$ is the constant defined at the beginning of the proof. 
Let $z_n \in \Ha$, such that $q_n=\iota(z_n)$. We may assume that $z_n=0$ and $q_n=q$, for some point $q$ that we fix. 
Then $p_n \to p$, where $d(p,q)\le M_0$. Moreover, since $\cl(\lambda_n) \to \infty$, the pleating maps $\iota(\lambda_n)$ converge to 
an isometry uniformly on compact sets in $\Ha$. In particular, the sequence of convex hulls $W_n$ converges to a geodesic plane uniformly 
on compact sets, and therefore $d(p_n,\partial_i{W_n}) \to 0$.  
But this contradicts (\ref{ass-2}), and thus we have completed the proof of the lemma.

\end{proof}

\subsection{$(\epsilon,R)$ skew pants}

We let $S$ continue to denote  a closed surface with a generalized pants decomposition $\Col$, and we fix a holomorphic   representations 
$\rho_{z,w}$ as in Proposition \ref{prop-rep}.

Let $\Col_0 \subset \Col$ denote a sub-collection of curves, and for each $C \in \Col_0$ we choose  
a number $-{{3}\over{4}}\pi < \theta_C<{{3}\over{4}} \pi$ (for each curve $C \in (\Col \setminus \Col_0)$ we set $\theta_C=0$).

For $C \in \Col$, let $\zeta_C, \eta_C \in \D$, where $\D$ denotes the unit disc in the complex plane.  Let $\tau\in \D$ 
denote a complex parameter and let $t \in \{0,1\}$.  Fix $R>1$, and let $z:\D \to \C^{\Col}_{+}$ and  $w:\D \to \C^{\Col}$ be the mappings given by
$$
z(C)(\tau)={{R}\over{2}}+ {{\tau \zeta_C}\over{2}}, 
$$
and 
$$
w(C)(\tau,t)=1+it \theta_C+{{\tau \eta_C} \over{R}}. 
$$
The maps $z(\tau)$ and $w(\tau,t)$ are complex linear, and therefore holomorphic in $\tau$ and $t$. Therefore the induced family of representations  $\rho_{\tau,t}=\rho_{z(\tau),w(\tau,t)}$ is 
holomorphic in $\tau$ and $t$. Note that $\rho_{\tau,t}$  depends on  $R$, $\zeta_C$, $\eta_C$ and  $\theta_C$, but we suppress this.

The representation $\rho_{0,0}$ is Fuchsian. Let $S_0$ denote the Riemann 
surface isomorphic to the quotient $\Ha / \rho_{0,0}(\pi_1(S))$   (we also equip $S_0$ with the corresponding hyperbolic metric). 
Let $\cl(\rho_{0,0})$ denote the largest number so that  the collection of collars (of width $\cl(\rho_{0,0})$) around the curves from $\Col_0$ is disjoint on  $S_0$.

The representation $\rho_{0,1}$ is not Fuchsian (unless $\theta(\Col_0)=0$), 
and the following proposition gives a sufficient condition for it to be quasifuchsian.

We adopt the following notation. Let $G(\tau,t)=\rho_{\tau,t}(\pi_1(S))$. If $G(\tau,t)$ is a quasifuchsian group we let  $f_{\tau,t}:\partial{\Ha} \to \partial{\Ho}$, 
denote the quasiconformal map that conjugates $G(0,0)$ to $G(\tau,t)$. The following theorem is a generalization of Theorem 2.2 from \cite{kahn-markovic-1} 
(see Theorem \ref{thm-geometry-old} below). Assuming the above notation, we have:

\begin{theorem} \label{thm-geometry} There exist universal constants $\wh{R},\wh{\epsilon}, M>0$, such that the following holds. 
If $\cl(\rho_{0,0})>M$, then for every $R \ge \wh{R}$ and $|\tau| <\wh{\epsilon}$, and any choice of constants $\eta_C,\zeta_C \in \D$, and   
$-{{3}\over{4}}<\theta_C<{{3}\over{4}}$, for $C \in \Col_0$,  the group $G(\tau,1)$  is quasifuchsian and the  induced quasiconformal map  $f_{\tau,1} \circ f_{0,1}$ 
(that conjugates $G(0,1)$  to  $G(\tau,1)$),  is  $K(\tau)$-quasiconformal, where 
$$
K(\tau)={{\wh{\epsilon} +|\tau|}\over{\wh{\epsilon}-|\tau|}}.
$$
\end{theorem}

Let $\Col_0(\tau,t)$ denote the collection of axes of  elements of the form $\rho_{\tau,t}(c)$, where $c \in \pi_1(S)$ and $c$ belongs to the conjugacy class 
of some curve $C \in \Col_0$. Then by definition, the set $\Col_0(\tau,t)$ is invariant under the group $G(\tau,1)$. Next, we prove that $\Col_0(\tau,1)$ is  invariant under any M\"obius transformation 
from $\PSLC$ that preserves the limit set of $G(\tau,1)$. The following theorem is the main result of this section.

\begin{theorem}\label{thm-rigid} There exist  constants  $\wh{\epsilon}_1, M_1>0$, with the following properties. Assume that $\cl(\rho_{0,0})>M_1$ and let $|\tau|<\wh{\epsilon}_1$. 
If  $T \in \PSLC$, is a M\"obius transformation  that preserves the limit set of $G(\tau,1)$, then the set of geodesics $\Col_0(\tau,1)$  is invariant under $T$.
\end{theorem}

Compare this theorem with Lemma 4.2 in \cite{masters}.

\begin{proof} Let $W(\tau,t)$ denote the convex hull of the limit set of $G(\tau,t)$. It follows from Lemma \ref{lemma-rigid} that for $\cl(\rho_{0,0})$ large enough, the following holds

\begin{enumerate}
\item For every $\gamma \in \Col_0(0,1)$ and $p \in \gamma$, the inequality $\max_{i=1,2} d(p,\partial_i{W(0,t)})> \delta_1$ holds, 
\item For every $\gamma \subset W(0,1)$ the inequality, there exists $p \in \gamma$ such that   $\max_{i=1,2} d(p,\partial_i{W(0,1)})< {{\delta_1}\over{2}}$.
\end{enumerate}

Then by Theorem \ref{thm-geometry} we can choose $\wh{\epsilon}_1$ small enough so that for $|\tau|<\wh{\epsilon}_1$, the constant $K(\tau)$ (from Theorem \ref{thm-geometry}) is close enough to $1$, 
so that the following holds:

\begin{enumerate}
\item For every $\gamma \in \Col_0(\tau,1)$ and $p \in \gamma$, the inequality $\max_{i=1,2} d(p,\partial_i{W(0,t)})> {{4\delta_1}\over{5}}$ holds, 
\item For every $\gamma \subset W(0,1)$ the inequality, there exists $p \in \gamma$ such that   $\max_{i=1,2} d(p,\partial_i{W(0,1)})< {{2\delta_1}\over{3}}$.
\end{enumerate}

Then any M\"obius transformation $A \in \PSLC$ that preserves $W(\tau,1)$ will also preserve the set $\Col(\tau,1)$. This proves the theorem.

\end{proof}

\subsection{A proof of Theorem \ref{thm-geometry}}  We need to  prove that $G(\tau,1)$ is a quasifuchsian group. The last estimate in Theorem \ref{thm-geometry} then follows from the fact that a 
holomorphic map from the unit disc into the Teichm\"uller space of a Riemann surface is a contraction with respect to the hyperbolic metric on the unit disc and the Teichm\"uller metric.

Recall  Theorem 2.2 from \cite{kahn-markovic-1}.

\begin{theorem} \label{thm-geometry-old}  There exist universal constants $\wh{R},\wh{\epsilon}$, such that the following holds. 
For every $R \ge \wh{R}$ and $|\tau| <\wh{\epsilon}$, and any choice of constants $\eta_C, \zeta_C \in \D$,  the group $G(\tau,0)$  
is quasifuchsian, and the  induced quasiconformal map  $f_{\tau,0}$ that conjugates $G(0,0)$  to  $G(\tau,0)$,  is  $K(\tau)$-quasiconformal, where 
$$
K(\tau)={{\wh{\epsilon} +|\tau|}\over{\wh{\epsilon}-|\tau|}}.
$$

\end{theorem}

The group $G(\tau,1)$ is obtained from the group $G(\tau,0)$, by bending along the lifts of curves $C \in \Col_0$, for the angle $\theta_C$. 
It follows from Theorem \ref{thm-sparsely} that the group $G(\tau,1)$ is quasifuchsian if $\cl(\rho_{0,0})>C$, and if the  map $f_{\tau,0}$ is $K$-quasiconformal, where $K$ is close enough to $1$. 
But it follows from Theorem \ref{thm-geometry-old}  that for 
$|\tau|$ small enough this will be the case. This proves Theorem \ref{thm-geometry}.

\section{The lower bound}

\subsection{Amalgamating two representations}  Let $S$ denote a closed surfaces with generalized pants decompositions $\Col$, and let $\rho:\pi_1(S) \to \PSLC$ denote an admissible (in sense of Definition \ref{def-1}) representation with the reduced Fenchel-Nielsen coordinates satisfying the inequalities
$$
|\hl(C)-{{R}\over{2}}| \le \epsilon,
$$
and
$$
|s(C)-1|\le {{\epsilon}\over{R}},
$$
for some $\epsilon, R>0$, and $C \in \Col$. We say that such a representation is $(\epsilon,R)$-good.

Let $\M$ denote a closed hyperbolic manifold such that $\M=\Ho /\Gamma$ for some Kleinian group $\Gamma$. In \cite{kahn-markovic-1} we proved that one can find many $(\epsilon,R)$-good representations 
$\rho:\pi_1(S) \to \Gamma$, for a given  $\epsilon>0$ and $R$ large enough. Moreover, if $A \in \Gamma$ has the translation length $\len(A)$  satisfying the inequality $|\len(A)-R|\le {{\epsilon}\over{2}}$, then we can find such $\rho$ so that $A$ is in the image of $\rho$. From now on we assume that such $A \in \Gamma$ is primitive, that is $A$ is not equal to an integer power of another element of $\Gamma$.

In particular, it follows from Section 4 of  \cite{kahn-markovic-1}  (the statements about the equidistribution of $(\epsilon,R)$-good pairs of skew pants around a given 
closed curve in $\M$ whose length is $\epsilon$ close to $R$) that we can find  two $(\epsilon,R)$-good representations $\rho(i):\pi_1(S(i)) \to \Gamma$, $i=1,2$, 
where $S(1)$ and $S(2)$ are two closed surfaces with pants decompositions 
$\Col(i)$, and two pars of pants $\Pi^{+}_i$ and $\Pi^{-}_i$  with the following properties:  

\begin{itemize}
\item There are two oriented, degree one  curves $C(i) \in \Col(i)$, and $c(i) \in \pi_1(S(i))$ in the conjugacy classes of $C(1)$ and $C(2)$ respectively, such 
that $\rho(1)(C(1))=\rho(2)(C(2))=[A]$, where $[A]$ is the conjugacy class of a given primitive element  $A \in \Gamma$, whose  translation length $\len(A)$  satisfies the inequality 
$|\len(A)-R|\le {{\epsilon}\over{2}}$. 
\item Let $\gamma$ denote the closed geodesic corresponding to   $A$. There exist  two pars of skew pants $\Pi^{+}_i$ and $\Pi^{-}_i$ in  $\rho(i)(\pi_1(S(i)))$ such that $\gamma$ is positively 
oriented boundary  component of $\Pi^{+}_i$ and negatively oriented for $\Pi^{-}_i$, and  recalling the notation from  \cite{kahn-markovic-1}  we have  the inequality

\begin{equation}\label{est-ang}
|\foot_{\gamma}(\Pi^{+}_2)-\foot_{\gamma}(\Pi^{-}_1) - {{\pi}\over{2}}| \le {{\epsilon}\over{R}}. 
\end{equation}

\end{itemize}

After replacing $S(1)$ and $S(2)$ with appropriate finite degree covers if necessary, we may assume in addition to the above two conditions the following also hold

\begin{itemize}

\item The curves $C(1)$ and $C(2)$ are non-separating simple closed curves in $S(1)$ and $S(2)$ respectively,
\item The surfaces $S(1)$ and $S(2)$ have the same genus,
\item  By Proposition \ref{prop-rep} the representation $\rho(i)$ can be embedded in the holomorphic family of representations $\rho_{\tau,t}(i)$. We may assume that  
$\cl(\rho_{0,0}(S(i)))>C_1$, $i=1,2$, where $C_1$ is the constant from Theorem \ref{thm-rigid}.

\end{itemize}

We now fix such two representations $\rho(1)$ and $\rho(2)$, surfaces $S(1)$ and $S(2)$, and the two oriented curves $C(1)$ and $C(2)$ (we also fix the corresponding primitive element $A \in \Gamma$).

Let $i \in \{1,2\}$. For $n>1$, let $S_n(1)$ and $S_n(2)$ denote two primitive degree $n$ covers of $S(1)$ and $S(2)$ respectively  (a finite cover of a surface is primitive if it does 
not factor through an intermediate cover), such that for some 
$1 \le k \le (n-1)$,   the curves $C(1)$ and $C(2)$ have two degree $k$ lifts $C_n(1)$ and $C_n(2)$. Then  $C_n(1)$ and $C_n(2)$ are two oriented, non-separating simple closed curves in $S_n(1)$ and $S_n(2)$ respectively. We then have the two induced representations  $\rho_n(i):\pi_1(S_n(i)) \to \Gamma$, that also satisfy the above five conditions, except that 
$$
\rho_n(1)(\pi_1(S_n(1))) \cap \rho_n(2)(\pi_1(S_n(2)))=\{A^{k}\}.
$$

We amalgamate them as follows. Cut the surface $S_n(i)$ along $C_n(i)$, to get two  topological surfaces $\overline{S}_n(i)$, $i=1,2$,  
each having two boundary components $C^{1}_n(i)$ and $C^{2}_n(i)$.  We glue together the surfaces $\overline{S}_n(1)$ and  $\overline{S}_n(2)$ by 
gluing $C^{j}_n(1)$ to $C^{j}_n(2)$, $j=1,2$,  and obtain a closed topological surface $S_n$ (this is well defined up to  a twist by $\Re (\len(A))$ which has a period $k$).
The surface $S_n$ has the induced generalized pants decomposition $\Col_n$. The  pair of curves $C^{1}_n(1)$ and $C^{1}_n(2)$ that were glued together  produce a closed curve 
$C^{1}_n$ in $S_n$. Similarly,  the pair of curves $C^{2}_n(1)$ and $C^{2}_n(2)$ that were glued together produce a closed curve $C^{2}_n$ in $S_n$. 
We set $\Col_{0,n}=\{C^{1}_{n},C^{2}_{n}\}$.

Then there is the induced representation $\rho_n:\pi_1(S_n) \to \Gamma$. We orient the curves $C^{1}_{n}$ and $C^{2}_{n}$ such that for any choice of $c_i \in \pi_1(S_n)$, 
where $c_i$ is in the conjugacy class of $C^{i}_n$, we have that both $\rho_n(c_1)$ and  $\rho_n(c_2)$ are in the conjugacy class of $A^{k}$ in $\Gamma$. 

The  representation $\rho_n$ has the reduced Fenchel-Nielsen coordinates satisfying the relations 
$$
|\hl(C)-{{R}\over{2}}| \le \epsilon,
$$
and
$$
|s(C)-1|\le {{\epsilon}\over{R}},
$$
if $C$ does not belong to $\Col_{0,n}$, and 
$$
|s(C)-(1+i{{\pi}\over{2}})|  \le {{\epsilon}\over{R}},
$$
if $C \in \Col_{0,n}$.

It follows from Theorem \ref{thm-geometry} that for $\epsilon$ small enough and $R$ large enough, the group $\rho_n(\pi_1(S_n))$ is  quasifuchsian. 
In the remainder of this subsection we prove that the group  $\rho_n(\pi_1(S_n))$ is a maximal subgroup of $\Gamma$.

First we prove a preliminary lemma.  Let $\overline{S}$ be a surface with boundary components $C_{+}$ and $C_{-}$, oriented so that $\overline{S}$ is on the left of $C_{+}$
and the right of $C_{-}$. We say that $f:\overline{S} \to \M$ is rejoinable if the restrictions of $f$ to $C^{+}$ and $C_{-}$ respectively are freely homotopic in $\M$. 
We say $(f,\overline{S})$ is geodesically rejoinable if $f|_{C_{+}}$ and $f|_{C_{-}}$ map to the same closed geodesic in $\M$. In this case we say a rejoining of $(f,\overline{S})$
is a homeomorphism $h:C_{+} \to C_{-}$ such that $f \circ h=f$, and we say $(f,\overline{S}/h)$ is $\overline{S}$ rejoined by $h$. 

\begin{lemma}\label{lemma-A} If $(f,\overline{S})$, and $(g,\overline{T})$ are (geodesically) rejoinable surfaces, and $\pi:\overline{S} \to \overline{T}$ is a finite cover such that 
$g \circ  \pi$ is homotopic to $f$ , then for any rejoining $h$ of $(f,\overline{S})$ we can find a rejoining $k$ of $(g,\overline{T})$ such that $(f,\overline{S})$ rejoined by $h$ covers 
$(g,\overline{T})$ rejoined by $k$.
\end{lemma}
\begin{proof} Left to the reader.

\end{proof}

The following theorem is a corollary of Theorem \ref{thm-rigid}. We adopt the following definition. Let 
$f:S \to \M$ be a quasifuchsian map, and let $\Col_0$ denote  a collection of disjoint simple closed curves 
on $S$. We say  that $f$ is bent along each curve of $\Col_0$ and nearly locally isometric on $S \setminus \Col_0$ if
the induced map $f_*:\pi_1(S) \to \Gamma$ is of the form $\rho_{\tau,1}$ for some $|\tau| \le \wh{\epsilon}$.

\begin{theorem}\label{thm-A} Let $S$ be a closed surface. Suppose that $f:S \to \M$ is a $\pi_1$-injective and quasifuchsian, and $\Col_0$ 
is a collection of disjoint simple closed curves 
on $S$, such that $f$ is bent along each curve of $\Col_0$ and nearly locally isometric on $S \setminus \Col_0$. Suppose that $f=g \circ \pi$, 
where $\pi:S \to Q$ is a covering, and $g:Q \to \M$ is 
$\pi_1$-injective  and quasifuchsian. Then we can find a collection of simple closed curves $\wh{\Col}_0$ on $Q$ such that $\Col_0=\pi^{-1}(\wh{\Col}_0)$.
\end{theorem}

\begin{proof} We get a discrete lamination $\wt{\Col}_0$ on $\Ha$, which we push forward by $\wt{f}=\wt{g}$ to $\Ho$. We find a homomorphism 
$\sigma:\Deck(\Ha / Q) \to \Gamma$
such that $\wt{f}(\gamma (x))=\sigma(\gamma)(\wt{f}(x))$ for every $x \in \Ha$ and $\gamma \in \Deck(\Ha /Q)$. 

We let $G=\sigma(\Deck(\Ha /Q))$, and $H=\sigma(\Deck(\Ha / S))<G$. Then $[G:H]<\infty$, and $G$ and $H$ are quasifuchsian groups, and 
they have the same limit set, so by Theorem \ref{thm-rigid}
every element of $G$ maps $\wt{g}(\wt{\Col}_0)$ to itself. Hence $\Deck(\Ha / Q)$ maps $\wt{\Col}_0$ to itself, so $\wt{\Col}_0$ 
is a lift of $\wh{\Col}_0$ on $Q$, and hence $\Col_0$ is.

\end{proof}

\begin{theorem} The quasifuchsian  group $\rho_n(\pi_1(S_n)) <\Gamma$ is a maximal surface subgroup of $\Gamma$, 
that is, if $\rho_n(\pi_1(S_n))<G$ for a surface subgroup $G< \Gamma$, 
then $G=\rho_n(\pi_1(S_n))$.
\end{theorem}

\begin{proof} 

For simplicity let $G_n=\rho_n(\pi_1(S_n))$ and $G(1)=\rho(1)(\pi_1(S(1)))$. Also set $G_n(1)=\rho_n(\pi_1(\overline{S}_n(1)))$, where we consider 
$\pi_1(\overline{S}_n(1))$ as a subgroup of $\pi_1(S_n)$.

Let $f_n:S_n \to \M$ denote the continuous map that corresponds to the representation $\rho_n$. 
We claim that $f_n:S_n \to \M$ is primitive. If not, we can find a Riemann surface $Q$ and $\pi:S_n \to Q$ and $g:Q \to \M$ such that 
$g \circ \pi=f_n$ and $d>1$ where $d$ is the degree of the cover $\pi$.
We recall that $f_n$ is bent along $C^{1}_{n}$ and $C^{2}_n$, and nearly  isometric on the complement.  
So by Theorem \ref{thm-A}, $\{C^{1}_n,C^{2}_n \}$ are the lifts by $\pi$ of some set $\Col_Q$ 
of simple closed curves on $Q$. So $|\Col _Q|=1$ or $|\Col_Q|=2$. 

If    $|\Col_Q|=2$, then each component of $S_n \setminus \cup C^{i}_n $ maps by degree $d$ to a component of $Q \setminus \Col_Q$. 
We can then write $Q \setminus \Col_Q=\overline{Q} (1) \cup \overline{Q}(2)$
such that $\pi:\overline{S}_n(i) \to \overline{Q}(i)$ is a degree $d$ cover, and  then by Lemma \ref{lemma-A} we can rejoin the boundary 
curves of $\overline{Q}(1)$ to form $Q'(1)$ such 
that $S_n(1)$ is a cover of $Q'(1)$.  But then we get a subgroup  $G_{Q'}$ of $G_n(1)$ ( $G_{Q'}=\pi_1(Q'(1))$), and 
$G_n(1) < G_{Q'} \cap G(1) < G(1)$, where both inclusions are proper. The first inclusion is proper because 
$A^{{{k}\over{d}}} \in G_{Q'} \cap G(1) \setminus G_n(1)$, and the second is proper 
because $k <n$. This contradicts the assumption on the maximality of $G_n(1)$.

If  $|\Col_Q|=1$, we let $\Col_Q=\{C_Q\}$.  First suppose that $C_Q$ is non-separating. Then writing $Q \setminus C_Q=\overline{Q}$ we 
find that $\overline{S}_n(1)$ and   
$\overline{S}_n(2)$  are both degree ${{d}\over{2}}$ covers of $\overline{Q}$. But then we can reassemble $\overline{Q}$ to make $Q'$ 
(by Lemma \ref{lemma-A}) such that $S_n(1)$ is a degree 
${{d}\over{2}}$ cover of $Q'$, when ${{d}\over{2}} \le  k$. Then we arrive at a contradiction by the same reasoning as before.

Finally, suppose that $C_Q$ is separating. Then we can write $Q \setminus C_Q=\overline{Q}(1) \cup \overline{Q}(2)$ so that the restriction of $\pi$ to $\overline{S}_n(i)$ is a cover of 
$\overline{Q}(i)$. Then the conjugacy classes for $C^{1}_n$ and $C^{2}_n$, oriented as curves covered by the axis of $A$, are both in
$[A^{k}]$, but $C^{1}_n$ and $C^{2}_n$ both cover $C_Q$ with opposite orientations, so the conjugacy class for $C_Q$ must be both $[A^{l}]$ and $[A^{-l}]$, where  $l={{2k}\over{d}}$.
But then $B^{-1}A^{l}B=A^{-l}$ for some $B \in \Gamma$, which means that $B$ preserves the axis of $A$ and reverses its orientation; such $B$ would have a fixed point in $\Ho$, which is a 
contradiction. 

\end{proof}

\subsection{The lower bound}

We now proceed to prove the lower bound
\begin{equation}\label{est-lower}
(c_1 g)^{2g} \le s_1(\M,g),
\end{equation}
for  $g$ large enough, from Theorem \ref{thm-main}. 

By the above theorem  the  representation $\rho_n:\pi_1(S_n) \to \Gamma$,  is maximal. It remains to count the number of such representations.
Let $g_n$  denote  the genus of the surface $S_n$. If $g_0$ denotes the genus of the surfaces $S(1)$ and $S(2)$, we have  
$$
g_n= n(2g_0-1).
$$

Given a closed surface $S_0$, Let $m_n(S_0)$ denote the number of maximal degree $n$ covers of $S_0$. Let $C_0$ denote a simple closed and non-separating curve in $S_0$. For $1 \le k \le n$, by  $m_n(S_0,C_0,k)$ we denote the number of maximal $n$ degree covers of $S_0$ such that the curve $C_0$ has at least one lift of degree $k$. Clearly the number $m_n(S_0,C_0,k)$ does not depend on the choice of the simple closed and non-non-separating curve $C_0$, so we sometimes write $m_n(S_0,k)=m_n(S_0,C_0,k)$.

\begin{theorem}\label{thm-count} Let $g_0$ denote the genus of $S_0$. Then for $n$ large we have:

$$
m_n(S_0)=(n!)^{g_{0}-2} (1 + o(1)),
$$
where $o(1) \to 0$ when $n \to \infty$. Moreover, for some $1\le k \le (n-1)$, $k=k(n,g_0)$,  we have 

$$
m_n(S_0,k) > ((n-1)!)^{g_{0}-2} (1 + o(1)). 
$$

\end{theorem}

\begin{proof} The first equality directly follows from Corollary 3 and the formula in Section 4.4 in \cite{muller-puchta}, which shows that  a random finite cover of a closed surface is maximal. 
It remains to prove the second inequality.

Since

$$
\sum\limits_{k=1}^{n} m_n(S_0,k) \ge m_n(S_0),
$$
it follows that for some $1 \le k \le n$, the second inequality in the statement of the theorem holds. The following lemma implies that this inequality holds for some $1 \le k \le (n-1)$.

\begin{lemma} The inequality $m_n(S_0,1) \ge m_n(S_0,n)$, holds for every $n$.
\end{lemma}

\begin{proof} Let $C_0$ and $D_0$ be two simple closed and non-separating curves on $S_0$, that intersect exactly once. Let $S_n$ be a degree $n$ cover of $S_0$, such that the curve $C_0$ has a degree $n$ lift which we denote by $C_n$. Then $C_n$ is the only lift of $C_0$. We show that in this case, every lift of the curve $D_0$ is a degree one lift. Let $\wt{S}_0=S_0 \setminus C_0$ and $\wt{S}_n=S_n \setminus C_n$, denote the two surfaces each having exactly two boundary components. Then $\wt{S}_n$ covers $\wt{S}_0$, because $C_n$ is the only lift of $C_0$ to $S_n$. After removing the curve $C_0$ from $S_0$, the closed curve $D_0$ becomes an interval $I_0 \subset \wt{S}_0$, whose endpoints lie on different boundary components of $\wt{S}_0$. Therefore, every lift of $I_0$ to $\wt{S}_n$ is a degree one lift. This proves the statement.

Restricting to  the cases when $S_n$ is a maximal cover, yields the inequality $m_n(S_0,C_0,n) \le m_n(S_0,D_0,1)$. Since $m_n(S_0,C_0,k)= m_n(S_0,D_0,k)= m_n(S_0,k)$, it follows that  $m_n(S_0,1) \ge m_n(S_0,n)$, and we have proved the lemma.

\end{proof}

This proves the theorem.

\end{proof}

Now fix a large $n$ and choose $1 \le k \le (n-1)$ so that the second inequality in Theorem \ref{thm-count} holds. We then amalgamate any two maximal covers $S_n(1)$ and $S_n(2)$ along the curves $C_n(1)$ and $C_n(2)$ that are both $k$ degree lifts of the curves $C(1)$ and $C(2)$ respectively (there may be more than one such $k$ degree lift, but we choose arbitrarily). Then the corresponding group $\rho_n(\pi_1(S_n)) <\Gamma$ is maximal surface subgroup of $\Gamma$.  Combining the above formula for $g_n$ with the Theorem \ref{thm-count}, we derive the estimate (\ref{est-lower}) for some $c_1>0$.

\end{document}